\documentclass[final,3p,times]{elsarticle}

\usepackage{latexsym, amssymb, enumerate, amsmath}
\usepackage{amsthm, epsfig, graphicx, wrapfig, dsfont}
\usepackage{color}

\newtheorem{thm}{Theorem}[section]

\newtheorem{lem}[thm]{Lemma}

\theoremstyle{definition}

\theoremstyle{remark}

\numberwithin{equation}{section}

\begin{document}

\begin{frontmatter}


\title{Stochastic Convective Wave Equation in Two Space Dimension}


\author[a]{Sang-Hyeon Park}
\ead[a]{laniean@nims.re.kr}

\author[a]{Imbo Sim\corref{cor1}}
\ead[b]{imbosim@nims.re.kr}

\address[a]{National Institute for Mathematical Sciences, Jeonmin-dong 463-1, Yuseong-gu, 305-811 Daejeon , Republic of Korea}

\cortext[cor1]{Corresponding author.}


\begin{abstract}
We study the convective wave equation in two space dimension driven by spatially homogeneous Gaussian noise. The existence of the real-valued solution is proved by providing a necessary and sufficient condition of Gaussian noise source. Our approach is based on the mild solution of the convective wave equation which is constructed by Walsh's theory of martingale measures. H\"older continuity of the solution is proved by using Green's function and  Kolmogorov continuity theorem.
\end{abstract}



\end{frontmatter}

\section{Introduction}\label{intro}
The purpose of this work is to study the propagation of acoustic waves in the presence of a uniform flow, driven by spatially homogeneous Gaussian noise source. In the absence of the mean flow, the acoustic problem can be constructed by the classical wave equation. The wave equation with a Gaussian noise has been studied by many articles, for example \cite{D1999, DF1998, MS1999, W1986}. Especially, Dalang and Frangos proved H\"older continuity  of the solution to stochastic wave equation in two spatial dimensions by presenting a necessary and sufficient condition for a real-valued stochastic process (\cite{DF1998}).

The linearized Euler equations model aeroacoustic problems in the presence of a uniform flow. The equations support acoustic waves, which propagates with the speed of sound relative to the mean flow, and vorticity and entropy waves, which travels with the mean flow. The entropy waves can be ignored in an inviscid, homogeneous fluid which does not conduct heat. For such a mean flow, the linearized Euler equations reduce into a convected wave equation for the pressure field. The presence of a mean flow makes the mathematical treatment of the problem much more difficult, mainly due to the acoustic waves whose phase and group velocities have opposite signs \cite{BGH2010, DJ2006, HHT2003}.

In this paper, we will study the convective wave equation in two space dimension. Unfortunately, a space-time white noise $\dot{W}(t,x)$ is not adapted to wave problem with two dimensions (refer to \cite{DF1998} for details). In this reason, Authors of \cite{DF1998} studied stochastic wave equation under a spatially homogeneous Gaussian noise. Similarly, we consider the convective wave equation driven by Gaussian noise $F$ as follows:
\begin{equation}\begin{split}\label{spde1}
&\left(\frac{\partial }{\partial t}+M\cdot \nabla\right)^2 u-\Delta u = \dot{F}(t,x),\\
&u(0,x)=0, \;\;\;\;\; t>0,\;x\in\mathds{R}^2,\\
&\frac{\partial u}{\partial t}=0,
\end{split}\end{equation}
where $\dot{F}$ is the formal derivative of Gaussian random field $F$ whose covariance function is given by $f(\|x\|)$ and $M=(M_1, M_2)$ is a Mach vector.

We assume that $f$ is a continuous function, $f:\mathds{R}_{+}\rightarrow \mathds{R}_{+}$, which holds the condition, $\int_{0^{+}}rf(r)dr<\infty$ (refer to \cite{DF1998} for details). Since the Laplacian operator $\Delta$ and $f(\|x\|)$ is rotational invariant, i.e. for any orthogonormal matrix $\Psi$, $f(\|\Psi x\|)=f(\|x\|)$, we consider the subsonic case, $M=(M_1,0)$ and $0\leq M_1<1$. Note that the problem (\ref{spde1}) is reduced to the wave equation considered by \cite{DF1998} ($M_1=0$). Here, we use a notation $M_1=m$ for our convenience. The Green function of (\ref{spde1}) in the case of $\dot{F}(t,x)=\delta(t)\delta(x)$, where $\delta(\cdot)$ is the Dirac delta function is given by
\begin{equation}\begin{split}\label{GF}
G(t,x,m):=\frac{1_{\{t-(\rho(x)-\frac{mx_1}{1-m^2})\geq0\}}}{2\pi \sqrt{1-m^2}\sqrt{(t+\frac{mx_1}{1-m^2})^2-\rho^2(x)}},
\end{split}\end{equation}
where $\rho(x):=\sqrt{\frac{x^2_1}{(1-m^2)^2}+\frac{x^2_2}{(1-m^2)}}$.

This paper organized as follows. First, we prove the existence of the real-valued solution by providing necessary and sufficient condition of a covariance function $f$ in Section 2. We study H\"older continuity of the solution in Section 3.

The result of this paper can be extended to the nonlinear case ($\sigma(u)\dot{F}(t,x)$ instead of $\dot{F}(t,x)$). Since a nonlinear case is verified by Picard iteration scheme and Gronwall's lemma if $|\sigma(u)|\leq K(1+|u||)$ (i.e. $\sigma$ is a globally Lipschitz function), we omit the nonlinear case. In this article, all positive real constants are denoted by $C$ or $C_i$, $i=1,2,\cdots$ in this paper.

\section{Stochastic Convected Wave Equation}
According to \cite{DF1998, W1986}, the model problem (\ref{spde1}) has a distribution-valued solution. Let $\mathds{D}(\mathds{R}^3)$ be the topological vector space of function $\phi$ in $\mathcal{C}^{\infty}_0(\mathds{R}^3)$. In $\mathds{D}(\mathds{R}^3)$, the convergence $\phi_n\rightarrow \phi$ defined by as follows:\\
1. there is a compact subset $K$ of $\mathds{R}^3$ such that supp$(\phi_n -\phi)\subset K$, for all $n$\\
2. $\lim_{n\rightarrow\infty}D^{\alpha}\phi_n=D^{\alpha}\phi$ unifomly on $K$ for each multiindex $\alpha$.\\

Let $F=F(\phi), \phi\in\mathds{D}(\mathds{R}^3)$ be a $L^2$-valued mean zero Gaussian process with covariance functional,
\begin{equation*}
E[F(\phi_1)F(\phi_2)]=\int_{\mathds{R}_+\times\mathds{R}^2}\phi_1(t,x)f(\|x-y\|)\phi_2(t,y)\;dxdyds.
\end{equation*}
We formally write this as a form $E[\dot{F}(t,x)\dot{F}(s,y)]=\delta(t-s)f(\|x-y\|)$. According to \cite{DF1998}, F has a $\mathds{D}^{\prime}(\mathds{R}^3)$ valued version. We formally define a martingale measure, $F((0,t]\times A):=F(1_{(0,t]\times A}(s,x))$, where A is a element of the Borel sigma-algebra $B(\mathds{R}^2)$. Then there exists the solution $u$ as a distribution with support in $\mathds{R}_+\times\mathds{R}^2$ such that for all $\phi\in\mathds{D}(\mathds{R}^3)$,
\begin{equation*}
\langle u,\left(\frac{\partial }{\partial t}+M\cdot \nabla\right)^2 \phi-\Delta \phi \rangle= \dot{F}(\phi),
\end{equation*}
where $\langle u,\phi\rangle := \int_{\mathds{R}_+\times\mathds{R}^2}u(t,x)\phi(t,x)\;dxdt$.

Since $\mathds{D}^{\prime}(\mathds{R}^3)$ is too large class, Dalang and Fragos in \cite{DF1998} studied a real-valued solution of classical wave equation (m=0 case) by worthy martingale measures as a form, $\int_{(0,t]\times\mathds{R}^2}G(t-s,x-y,0)\;dF(s,y)$ (refer to \cite{DF1998, W1986} for details). According to \cite{DF1998}, the previous stochastic integration is well-defined and square integrable i.e. $E[|\cdot|^2]<\infty$.

By applying the similar argument to \cite{DF1998},  we will study a solution of (\ref{spde1}) as a real-valued process in the class of $\{X(t,x)|\; E[X(t,x)^2]<\infty\}$ as follows.
\begin{thm}\label{th1}
Let $u$ be the distribution-valued solution of (\ref{spde1}). Then there exists a jointly measurable process, $X(t,x)=\int_{(0,t]\times\mathds{R}^2}G(t-s,x-y,m)\;dF(s,y)$ (refer to \cite{DF1998, W1986} for details), which is square integrable such that
\begin{equation*}
\langle u,\phi\rangle = \int_{\mathds{R}_+\times\mathds{R}^2}X(t,x)\;\phi(t,x)\;dxdt\;\;\;\;a.s., \;\;\;\text{for all $\phi \in\mathds{D}(\mathds{R}^3)$},
\end{equation*}
if and only if
\begin{equation}\label{12073}
\int_{0^{+}}r\ln{\frac{1}{r}}f(r)\;dr<\infty.
\end{equation}
\end{thm}

{\bf Remark 1.} The function $f(\|x\|)=|x|^{-\alpha}$, $0<\alpha<2$ is usually applied to a Gaussian noise (for example, \cite{DM2010}). Clearly, this function satisfies (\ref{12073}). Therefore, the solution $u$ is not a distributed-valued but a real-valued stochastic process.\\

The proof of Theorem \ref{th1} needs following three lemmas.\\

\begin{lem}\label{lem0}
Let constants $a,b$, and C be positive. Suppose there exists a positive constant $\epsilon$ such that $0<c+\epsilon<a<b$. Then there exists a constant $C>0$ such that
\begin{equation}\begin{split}\label{lemin0}
\int^{b}_{a-\epsilon}&(s^2-c^2)^{-1/2}\left((s+\epsilon)^2-a^2\right)^{-1/2}ds\leq C \int^{b+\epsilon}_{a}(s^2-\tilde{c}^2)^{-1/2}(s^2-a^2)^{-1/2}ds,
\end{split}\end{equation}
where $\tilde{c}=c+\epsilon$.
\end{lem}
\begin{proof} See the proof in appendix.
\end{proof}

\begin{lem}\label{lem1}
Let a function $g:\Omega\subset\mathds{R}_{+} \rightarrow \mathds{R}_{+}$ be positive and not increasing. We define following subsets of $\mathds{R}^2\times\mathds{R}^2$
\begin{equation}\begin{split}
D_1(x,y)&=\{(x,y)|\;0<\rho(x)-\frac{m x_1}{1-m^2}<\rho(y)-\frac{m y_1}{1-m^2}<t, y_1>x_1\},\\
D_2(x,y)&=\{(x,y)|\;0<\rho(x)-\frac{m x_1}{1-m^2}<\rho(y)-\frac{m y_1}{1-m^2}<t, y_1<x_1, \rho(y)>\rho(x)\},\\
D_3(x,y)&=\{(x,y)|\;0<\rho(x)-\frac{m x_1}{1-m^2}<\rho(y)-\frac{m y_1}{1-m^2}<t, y_1<x_1, \rho(y)<\rho(x)\}.
\end{split}\end{equation}
For all $t>0$, three integrations,
\begin{equation}\begin{split}\label{lemin1}
&\int_{D_1(x,y)}\frac{f(\|y-x\|)}{\rho(y)}g(\rho^2(y)-\rho^2(x))\;dxdy,\\
&\int_{D_2(x,y)}\frac{f(\|y-x\|)}{\rho(y)}g(\rho^2(y)-\rho^2(x))\;dxdy,\\
\text{and}\;\;&\int_{D_3(x,y)}\frac{f(\|y-x\|)}{\rho(x)}g(\rho^2(x)-\rho^2(y))\;dxdy
\end{split}\end{equation}
are bounded by
\begin{equation}\label{lemin2}
\int^{2\sqrt{\frac{1+m}{1-m}}t}_{0}\int^{2(1+m)t}_{r\sqrt{1-m^2}}\int^{\frac{\pi}{2}}_{\eta^{-1}\left(\frac{w}{(1-m^2)r}\right)}rf(r)g\Big(\frac{rw\eta(\theta)}{1-m^2}-r^2\eta^2(\theta)\Big)\Big(\ln(4(1+m)t)-\ln{w}\Big)\; d\theta dw dr,
\end{equation}
where
\begin{equation}\label{eta}
\eta(\theta):=\sqrt{\frac{\cos^2{\theta}}{(1-m^2)^2}+\frac{\sin^2{\theta}}{1-m^2}}.
\end{equation}
\end{lem}
\begin{proof}
Let $z=y-x$ for fixed y . We define the subsets of $\mathds{R}^2\times\mathds{R}^2$
\begin{equation}\begin{split}
\tilde{D}_1(y,z)&=\{(y,z)|\;0<\rho(y-z)-\frac{m (y_1-z_1)}{1-m^2}<\rho(y)-\frac{m y_1}{1-m^2}<t, z_1>0\},\\
\tilde{D}_2(y,z)&=\{(y,z)|\;0<\rho(y-z)-\frac{m (y_1-z_1)}{1-m^2}<\rho(y)-\frac{m y_1}{1-m^2}<t, z_1<0, \rho(y)>\rho(y-z)\}.\\
\end{split}\end{equation}
Then the first and second term of (\ref{lemin1}) are equal to
\begin{equation}\begin{split}
\int_{\tilde{D}_1(y,z)}\frac{1}{\rho(y)}f(\|z\|)g(\rho^2(y)-\rho^2(y-z))\;dzdy,\\
\int_{\tilde{D}_2(y,z)}\frac{1}{\rho(y)}f(\|z\|)g(\rho^2(y)-\rho^2(y-z))\;dzdy,
\end{split}\end{equation}
respectively.

Let T be a transform $T:(y,z)\rightarrow (\nu,\theta_0,\tilde{r},\tilde{\theta})\;$ such that $y=(\nu\cos{\theta_0},\frac{\nu}{\sqrt{1-m^2}}\sin{\theta_0})$ and $z=(\tilde{r}\cos{(\tilde{\theta}-\theta_0}),\frac{\tilde{r}}{\sqrt{1-m^2}}\sin{(\tilde{\theta}-\theta_0}))$ for fixed y, $0<\tilde{\theta},\theta_0\leq 2\pi$ and $\nu, \tilde{r}>0$. Since $\{y|\;0<\rho(y)-\frac{m y_1}{1-m^2}<t\}\subset \{y|\;\rho(y)<\frac{t}{1-m}\}$, we obtain
\begin{equation*}\begin{split}
&T(\{(y,z)|\; \rho(y-z)<\rho(y)\})=\left\{(\nu,\theta_0,\tilde{r},\tilde{\theta})|\; \frac{\tilde{r}^2}{(1-m^2)^2}<\frac{2\tilde{r}\nu}{(1-m^2)^2}\cos{\tilde{\theta}}\right\},\\
&T(\{y|\; \rho(y)-\frac{m y_1}{1-m^2}<t\})\subset\{\nu|\;\nu<(1+m)t\}.
\end{split}\end{equation*}
These lead to
\begin{equation*}\begin{split}
T(\tilde{D}_1(y,z))&\subset\{(\nu,\theta_0,\tilde{r},\tilde{\theta})|\; \tilde{r}<2\nu\cos{\tilde{\theta}}\}\cap\{\nu|\; \nu<(1+m)t\}\cap\{(\theta,\tilde{\theta})|\;\cos{(\tilde{\theta}-\theta_0)}>0\}\\
&=\{(\nu,\theta_0,\tilde{r},\tilde{\theta})|\;0<\tilde{r}<2\nu, 0<\tilde{\theta}<\cos^{-1}\left(\frac{\tilde{r}}{2\nu}\right),\;\nu<(1+m)t,\;\cos{(\tilde{\theta}-\theta_0)}>0\}.
\end{split}\end{equation*}
Here, we use the fact $\tilde{D}_1(y,z)\subset\{(y,z)|\; \rho(y-z)<\rho(y),\; 0<\rho(y)-\frac{m y_1}{1-m^2}<t,\; z_1>0\}$. Therefore, by Fubini's theorem, the first term of (\ref{lemin1}) is bounded by
\begin{equation}\label{abe123412}
\int^{(1+m)t}_{0}\int^{2\pi}_{0}\int^{2\nu}_{0}\int^{\cos^{-1}{(\frac{\tilde{r}}{2\nu})}}_{0}\frac{\tilde{r}f(\|z\|)}{1-m^2}g\Big(\frac{2\tilde{r}\nu\cos{\tilde{\theta}}}{(1-m^2)^2}-\frac{\tilde{r}^2}{(1-m^2)^2}\Big) \;d\tilde{\theta} d\tilde{r}d\theta_0 d\nu.
\end{equation}
Now we set $w=2\nu\cos{\tilde{\theta}}$ , then (\ref{abe123412}) is equal to
\begin{equation}\label{abe34122}
\int^{2(1+m)t}_{0}\int^{w}_{0}\int^{2\pi}_{0}\left(\int^{(1+m)t}_{w/2}\frac{\tilde{r}f(\|z\|)}{(1-m^2)\sqrt{4\nu^2-w^2}}g\Big(\frac{\tilde{r}w-\tilde{r}^2}{(1-m^2)^2}\Big) \;d\nu \right) \;d\theta_0d\tilde{r}dw.
\end{equation}
By $\int^{(1+m)t}_{\frac{w}{2}}(4\nu^2-w^2)^{-1/2}d\nu=\ln{\Big(2\nu+2\sqrt{\nu^2-(\frac{w}{2})^2}\Big)}|^{(1+m)t}_{\frac{w}{2}}$, the integral (\ref{abe34122}) reduces to
\begin{equation}\begin{split}\label{abe17845}
\int^{2(1+m)t}_{0}\int^{w}_{0}\int^{2\pi}_{0}\frac{\tilde{r}f(\|z\|)}{(1-m^2)}g\Big(\frac{\tilde{r}w-\tilde{r}^2}{(1-m^2)^2}\Big)\Big(\ln(4(1+m)t)-\ln{w}\Big)\; d\theta_0d\tilde{r}dw.
\end{split}\end{equation}
Since $\frac{\tilde{r}}{1-m^2}=\rho(z)$ in (\ref{abe17845}), the integral (\ref{abe17845}) can be reformulated as follows.
\begin{equation}\begin{split}\label{abe1}
\int^{2(1+m)t}_{0}\int_{\{z|\;\rho(z)<\frac{w}{1-m^2}\}}f(\|z\|)g\Big(\frac{w\rho(z)}{1-m^2}-\rho^2(z)\Big)\Big(\ln(4(1+m)t)-\ln{w}\Big)\; dzdw.
\end{split}\end{equation}
The function $\eta(\theta)$ in (\ref{eta}) is $\pi$- periodic, $\frac{1}{\sqrt{1-m^2}}\leq \eta(\theta)\leq \frac{1}{1-m^2}$ and $\eta^{-1}$ exist on $[0,\frac{\pi}{2}]$. By imposing $z=(r\cos\theta,r\sin\theta)$ and Fubini's theorem, the integral (\ref{abe1}) can be rewritten as follows.
\begin{equation}\begin{split}\label{zzz12}
&4\int^{2(1+m)t}_{0}\int^{\frac{w}{\sqrt{1-m^2}}}_{0}\int^{\frac{\pi}{2}}_{\eta^{-1}\left(\frac{w}{(1-m^2)r}\right)}rf(r)g\Big(\frac{rw\eta(\theta)}{1-m^2}-r^2\eta^2(\theta)\Big)\Big(\ln(4(1+m)t)-\ln{w}\Big)\; d\theta dr dw\\
&=4\int^{2\sqrt{\frac{1+m}{1-m}}t}_{0}\int^{2(1+m)t}_{r\sqrt{1-m^2}}\int^{\frac{\pi}{2}}_{\eta^{-1}\left(\frac{w}{(1-m^2)r}\right)}rf(r)g\Big(\frac{rw\eta(\theta)}{1-m^2}-r^2\eta^2(\theta)\Big)\Big(\ln(4(1+m)t)-\ln{w}\Big)\; d\theta dw dr.
\end{split}\end{equation}

In a similar way, the second and third integral in $(\ref{lemin1})$ can be proved in the following. Let $z_1<0$, then we obtain
\begin{equation*}\begin{split}
&\{(y,z)|\;0<\rho(y-z)-\frac{m(y_1-z_1)}{1-m^2}<\rho(y)-\frac{my_1}{1-m^2}<t\}\cap \{(y,z)|\;\rho(y-z)<\rho(y)\}\\
&=\{(y,z)|\;0<\rho(y-z)-\frac{m(y_1-z_1)}{1-m^2}<\rho(y)-\frac{my_1}{1-m^2}\}\cap\{(y,z)|\;\rho(y-z)<\rho(y)\}\cap\{y|\;\rho(y)-\frac{my_1}{1-m^2}<t\}\\
&\subset\{(y,z)|\;\rho(y-z)<\rho(y)\}\cap\{y|\;0<\rho(y)-\frac{my_1}{1-m^2}<t\}\\
&\subset \{(y,z)|\;\rho(y-z)<\rho(y)\}\cap\{y|\;\rho(y)<\frac{t}{1-m}\}.
\end{split}\end{equation*}
Therefore, the second integral in (\ref{lemin1}) is bounded by (\ref{zzz12}).

Moreover, if $y_1<x_1$,
\begin{equation*}\begin{split}
&\{(x,y)|\;0<\rho(x)-\frac{mx_1}{1-m^2}<\rho(y)-\frac{my_1}{1-m^2}<t\}\cap \{(x,y)|\;\rho(x)>\rho(y)\}\\
&\subset \{(x,y)|\;\rho(x)>\rho(y)\}\cap\{x|\;0<\rho(x)-\frac{mx_1}{1-m^2}<t\}.
\end{split}\end{equation*}
In consequence, the third integral in (\ref{lemin1}) is bounded by
\begin{equation}
\int_{\{(x,y)|\;\rho(x)>\rho(y),y_1<x_1,0<\rho(x)-\frac{mx_1}{1-m^2}<t\}}\frac{f(\|y-x\|)}{\rho(x)}g(\rho^2(x)-\rho^2(y))\;dxdy.
\end{equation}
Let $z=x-y$ for fixed x. Let $\bar{T}$ be a transform $\bar{T} : (x,y))\rightarrow (\nu,\theta_0,\tilde{r},\tilde{\theta})$ such that $x=(\nu\cos{\theta_0},\frac{\nu}{\sqrt{1-m^2}}\sin{\theta_0})$ and $z=(\tilde{r}\cos{(\tilde{\theta}-\theta_0}),\frac{\tilde{r}}{\sqrt{1-m^2}}\sin{(\tilde{\theta}-\theta_0}))$ for fixed x, $0<\tilde{\theta},\theta_0\leq 2\pi$ and $\nu, \tilde{r}>0$. Then $\bar{T}(\{(x,y)|\;\rho(x)>\rho(y),y_1<x_1,0<\rho(x)-\frac{mx_1}{1-m^2}<t\})\subset\{(\nu,\theta_0,\tilde{r},\tilde{\theta})|\;0<\tilde{r}<2\nu, 0<\tilde{\theta}<\cos^{-1}\left(\frac{\tilde{r}}{2\nu}\right),\;\nu<(1+m)t,\;\cos{(\tilde{\theta}-\theta_0)}>0\}$.
In doing so, the third integration of (\ref{lemin1}) is also bounded by (\ref{zzz12}).
\end{proof}

\begin{lem}\label{lem3}
Let a function $g:\Omega\subset\mathds{R}_{+} \rightarrow \mathds{R}_{+}$ be positive and not increasing. Then, for small $t>0$
\begin{equation}\begin{split}\label{lemin51}
&\int_{\hat{D}_1(x,y)}f(\|y-x\|)g(\rho^2(x)-\rho^2(y))\;dxdy\geq C\int^{2c_1 t}_{0}rf(r)\int^{2c_1 t}_{\frac{r}{1-m^2}} g(\frac{rw}{1-m^2}-\frac{r^2}{(1-m^2)^2})\sqrt{4c^2_1 t^2-w^2}\;dwdr,
\end{split}\end{equation}
where $\hat{D}_1(x,y)=D_1(x,y)\cup\{(x,y)|\rho(x)<\rho(y)\;\}$ and $C$ is a positive constant depending on the Mach number $m$.
\end{lem}
\begin{proof} Let $z=y-x$ for fixed $y$. Since $\{(y,z)|\;\rho(y-z)<\rho(y)<\frac{t}{1+m},0<z_1\}\subset\hat{D}_1(y-z,y)$, the left-hand side integral in (\ref{lemin51}) is greater than
\begin{equation}\begin{split}\label{eq29456}
\int_{A_1}f(\|z\|)g\left(\rho^2(y)-\rho^2(y-z)\right)\;dzdy,
\end{split}\end{equation}
where $A_1=\{(y,z)|\;\rho(y-z)<\rho(y)<\frac{t}{1+m}, 0<z_1\}$. Let T be a transform $T:(y,z)\rightarrow(\nu,\theta_0,r,\theta)$ such that $y=(\nu\cos{\theta_0},\frac{1}{1-m^2}\nu\sin{\theta_0})$ and $z=(r\cos(\theta-\theta_0),r\sin(\theta-\theta_0))$. Then $\{(\nu,\theta_0,r,\theta)|\;\frac{r}{(1-m^2)}<2\nu\cos{\theta}\}\subset T(\{(y,z)|\;\rho(y-z)<\rho(y)\})$ and $\{\nu|\;\nu<(1+m)t\}\subset T(\{y|\;\rho(y)<\frac{t}{1-m}\})$. Therefore, the integral (\ref{eq29456}) is greater than
\begin{equation*}\begin{split}
&\int_{T^{-1}(\tilde{A}_1)}f(\|z\|)g\left(\rho^2(y)-\rho^2(y-z)\right)\;dzdy\\
&\geq C\int^{(1+m)t}_{0}\int^{2(1-m^2)\nu}_{0}\int^{2\pi}_{0}\int^{\cos^{-1}{\frac{r}{2\nu(1-m^2)}}}_{0}\nu rf(r) 1_{\{\cos{(\theta-\theta_0)}>0\}}\tilde{g}(\nu,\theta_0,r,\theta)\;d\theta d\theta_0 dr d\nu,
\end{split}\end{equation*}
where $\tilde{A}_1=\{(\nu,\theta_0,r,\theta)|\;\frac{r}{(1-m^2)}<2\nu\cos{\theta}, \nu<(1+m)t, \cos{(\theta-\theta_0)}>0\}$ and $\tilde{g}(\nu,\theta_0,r,\theta)=g(\rho^2(y)-\rho^2(y-z))$. Let $w=2\nu\cos{\theta}$. Since the function $g$ is not increasing and the non-empty set $\{(\theta,\theta_0)|\;\cos{(\theta-\theta_0)}>0\}$ is bounded, we have a lower bound,
\begin{equation*}\begin{split}
&C\int^{(1+m)t}_{0}\int^{2\|y\|}_{0}rf(r)\int^{2\|y\|}_{\frac{r}{1-m^2}} \frac{\nu g(\rho^2(y)-\rho^2(y-z))}{\sqrt{4\nu^2-w^2}}\;dwdr d\nu\\
&\geq C\int^{2(1+m) t}_{0}rf(r)\int^{2(1+m) t}_{\frac{r}{1-m^2}} g(\rho^2(y)-\rho^2(y-z))\sqrt{4(1+m)^2t^2-w^2}\;dwdr.
\end{split}\end{equation*}
\end{proof}

${}$\\
{\bf Proof of Theorem \ref{th1}}
\begin{proof} {\bf (proof of necessity) } By the same way to Theorem 1 in Dalang~\cite{DF1998} (refer to pages $199-200$ of \cite{DF1998}), it is enough to show that for fixed $0\leq m <1$ and a small $t>0$
\begin{equation}\begin{split}\label{eqnece1}
\int^{t}_{0}&\int_{\mathds{R}^2}\int_{\mathds{R}^2}G(t-s,x-\tilde{y},m)G(t-s,x-\tilde{z},m)f(\|\tilde{y}-\tilde{z}\|)\;d\tilde{y}d\tilde{z}ds\geq C\int^{t}_{0}rf(r)\ln{\frac{1}{r}}dr,
\end{split}\end{equation}
where $C$ is a positive constant depending on the March number $m$. Let $\tilde{s}=t-s$, $y=x-\tilde{y}$, $z=x-\tilde{z}$. Then, by Fubini's theorem, the left-hand side of (\ref{eqnece1}) is greater than
\begin{equation}\begin{split}\label{eqnece2}
\int_{\hat{D}_1(y,z)}f(\|y-z\|)\int^{t}_{\rho(y)-\frac{m y_1}{1-m^2}}G(\tilde{s},y,m)G(\tilde{s},z,m)\;d\tilde{s}dzdy.
\end{split}\end{equation}
Set $\bar{s}=\tilde{s}+\frac{m y_1}{1-m^2}$. By using the fact $s^2-\rho^2(z)\geq(s+\frac{m(z_1-y_1)}{1-m^2})^2-\rho^2(z)$ in $\{(y,z)|\;y_1>z_1\}$, the integral (\ref{eqnece2}) is greater than
\begin{equation}\begin{split}\label{eqnece3}
C\int_{D_1(y,z)}f(\|y-z\|)\int^{t+\frac{m y_1}{1-m^2}}_{\rho(y)}\left(\bar{s}^2-\rho^2(y)\right)^{-1/2}\left(\bar{s}^2-\rho^2(z)\right)^{-1/2}d\bar{s}dzdy.
\end{split}\end{equation}
Let $\tau=s^2$. Since $\int(s^2+as+b)^{-1/2}ds=\ln{(a+2s+2\sqrt{s^2+as+b})}$, we can check that the integral (\ref{eqnece3}) is greater than
\begin{equation*}\begin{split}
C\int_{\hat{D}_1(y,z)}&f(\|y-z\|)\left(\ln{\Big(\sqrt{(t+\frac{my_1}{1-m^2})^2-\rho^2(y)}+\sqrt{(t+\frac{my_1}{1-m^2})^2-\rho^2(z)}\Big)^2}-\ln{\left(\rho^2(y)-\rho^2(z)\right)}\right)dzdy.
\end{split}\end{equation*}
From the condition $\rho(y)-\frac{my_1}{1-m^2}<t$ in $D_1(y,z)$, we obtain
\begin{equation}\label{ttt}
\frac{t}{1+m}\leq t+\frac{my_1}{1-m^2}\leq\frac{t}{1-m}
\end{equation}
Therefore, we have a lower bound,
\begin{equation}\begin{split}\label{eq1111}
C&\int_{\hat{D}_1(y,z)}f(\|y-z\|)\ln{\Big(2\frac{t^2}{(1-m)^2} -\rho^2(y)-\rho^2(z)\Big)}dzdy-C\int_{\hat{D}_1(y,z)}f(\|y-z\|)\ln{\left(\rho^2(y)-\rho^2(z)\right)}\;dzdy.
\end{split}\end{equation}
For small $t$, we obtain $\ln{\Big(2\frac{t^2}{(1-m)^2} -\rho^2(y)-\rho^2(z)\Big)}<0$ and $\ln{\left(\rho^2(y)-\rho^2(z)\right)}<0$ in $D_1(y,z)$. Hence, the first integration of (\ref{eq1111}) is greater than
\begin{equation*}\begin{split}
\int_{\hat{D}_1(y,z)}f(\|y-z\|)\ln{\Big(\frac{t^2}{(1-m)^2} -\rho^2(y)\Big)}\;dzdy.
\end{split}\end{equation*}
Now, we can check that $\{(y,z)|\;\rho(z)<\rho(y)<\frac{t}{1+m}, z_1<y_1\}\subset\hat{D}_1(y,z)$. Let $\bar{z}=y-z$ for fixed $y$. By using of polar coordinates $\bar{z}=(r\cos\theta,r\sin\theta)$, the above integration is greater than
\begin{equation*}\begin{split}
C\int_{\rho(y)<\frac{t}{1+m}}\ln{\Big(\frac{t^2}{(1-m)^2} -\rho^2(y)\Big)}\int^{\frac{t}{1+m}}_{0}rf(r)\;drdy>-\infty.
\end{split}\end{equation*}
By Lemma \ref{lem3}, the second integral in (\ref{eq1111}) is greater than
\begin{equation*}\begin{split}
&C\int^{(1+m) t}_{0}rf(r)\ln{\frac{1}{r}}\int^{(1+m) t}_{\frac{r}{1-m^2}}\sqrt{((1+m) t)^2 -w^2}\; dwdr\\
&-C\int^{(1+m) t}_{0}rf(r)\int^{(1+m) t}_{\frac{r}{1-m^2}}\sqrt{((1+m) t)^2 -w^2}\ln{\Big(\frac{w}{(1-m^2)^2}-\frac{r}{1-m^2}\Big)}\;dwdr\\
&\geq C\int^{(1+m) t}_{0}rf(r)\ln{\frac{1}{r}}\int^{(1+m) t}_{\frac{r}{1-m^2}}\sqrt{((1+m) t)^2 -w^2} \;dwdr\\
&\geq C\int^{(1+m) t}_{0}rf(r)\ln{\frac{1}{r}}dr.
\end{split}\end{equation*}
Here, we use the inequality $\sqrt{((1+m) t)^2 -w^2}\ln{\left(\frac{w}{(1-m^2)^2}-\frac{r}{1-m^2}\right)}<0$ for small $t$. Therefore, we finish the proof of necessity.\\

{\bf (proof of sufficiency) :} let $X(t,x)=\int^{t}_{0}\int_{\mathds{R}^2}G(t-s,x-y,m)F(ds,dy)$, for $t>0$ and $x\in\mathds{R}^2$. It is enough to show that $E[|X(t,x)|^2]<\infty$, if $\int_{0^{+}}rf(r)\ln{\frac{1}{r}}dr<\infty$. By definition of $X(t,x)$ in Theorem \ref{th1}, we obtain
\begin{equation}\begin{split}\label{ghighi}
E[|X(t,x)|^2]=2 \int_{\cup^{3}_{i=1}D_i(y,z)}f(\|y-z\|)\int^{t}_{\rho(y)-\frac{m y_1}{1-m^2}}G(\tilde{s},y,m)G(\tilde{s},z,m)\;d\tilde{s}dzdy.
\end{split}\end{equation}
Case 1 ($z_1<y_1$) : by setting $s=\tilde{s}+\frac{m z_1}{1-m^2}$, the expectation $E[|X(t,x)|^2]$ is bounded by
\begin{equation}\begin{split}\label{abe001}
&\int_{D_1(y,z)}f(\|y-z\|)\int^{t+\frac{m z_1}{1-m^2}}_{\rho(y)-\frac{m (y_1-z_1)}{1-m^2}}\left(s^2-\rho^2(z)\right)^{-1/2}\left(\left(s+\frac{m (y_1-z_1)}{1-m^2}\right)^2-\rho^2(y)\right)^{-1/2}dsdzdy.
\end{split}\end{equation}
Set $\epsilon=\frac{m(y_1-z_1)}{1-m^2}$. It satisfies the assumption of Lemma \ref{lem0}, since $\epsilon<\rho(y)-\rho(z)$ and $\rho(z)<\rho(y)$. By Lemma \ref{lem0}, the integral (\ref{abe001}) is bounded by
\begin{equation}\begin{split}\label{eqsu1}
&\int_{\tilde{D}_1(y,z)}f(\|y-z\|)\int^{t+\frac{m y_1}{1-m^2}}_{\rho(y)}\left(s^2-\left(\rho(z)+\frac{m(y_1-z_1)}{1-m^1}\right)^2\right)^{-1/2}\left(s^2-\rho^2(y)\right)^{-1/2}dsdzdy\\
&=\int^{\frac{t}{1-m}}_{0}\int_{\rho(y)<s}\int_{A_2}f(\|y-z\|)\left(s^2-\left(\rho(z)+\frac{m(y_1-z_1)}{1-m^1}\right)^2\right)^{-1/2}\left(s^2-\rho^2(y)\right)^{-1/2}dzdyds,
\end{split}\end{equation}
where $A_2=\{z|\;0<\rho(z)-\frac{mz_1}{1-m^2}<\rho(y)-\frac{my_1}{1-m^2}\}\cap\{z|\;y_1>z_1\}$. Set $A_{2,1}=\{z|\;0<\rho(z)-\frac{mz_1}{1-m^2}<\rho(y)-\frac{my_1}{1-m^2}\}\cap\{z|\;0<y_1-z_1<\tilde{r}_0\}$ and $A_{2,2}=\{z|\;0<\rho(z)-\frac{mz_1}{1-m^2}<\rho(y)-\frac{my_1}{1-m^2}\}\cap\{z|\;\tilde{r}_0<y_1-z_1\}$, where $\tilde{r}_0>0$. Then the right-hand side of (\ref{eqsu1}) can be rewritten as follows:
\begin{equation}\begin{split}\label{eq1234}
&\int^{\frac{t}{1-m}}_{0}\int_{\rho(y)<s}\int_{A_{2,1}} \cdot \;dzdyds+\int^{\frac{t}{1-m}}_{0}\int_{\rho(y)<s}\int_{A_{2,2}} \cdot \;dzdyds.
\end{split}\end{equation}
Here the integrand is omitted for the convenience. Now, we choose the constant $\tilde{r}_0$ which satisfies $\frac{s^2-\rho^2(z)}{s^2-\left(\rho(z)+\frac{m(y_1-z_1)}{1-m^2}\right)^2}<4$ for all $z\in \{z|\;0<y_1-z_1<\tilde{r}_0\}$. Then the first integral of (\ref{eq1234}) is bounded by
\begin{equation*}\begin{split}
\int_{D_1(y,z)}f(\|y-z\|)\int^{t+\frac{m y_1}{1-m^2}}_{\rho(y)}\left(s^2-\rho^2(z)\right)^{-1/2}\left(s^2-\rho^2(y)\right)^{-1/2}dsdzdy.
\end{split}\end{equation*}
Since $f$ is positive and continuous in $\mathds{R}_{+}$, there exist positive constants, $C_1=\mathop{\max}_{x\in[\tilde{r}_0,\frac{2t}{1-m}]}\{f(x)\}$ and $C_2=\mathop{\min}_{x\in[\tilde{r}_0,\frac{2t}{1-m}]}\{f(x)\}$.
We define two curves in $\mathds{R}^2$, $L_1=\{z|\;\rho(z)+\frac{m (y_1-z_1)}{1-m^2}=\alpha\}$ and $L_2=\{z|\;\rho(z)=\alpha\}$ for $\alpha>0$. Then, by Fubini's theorem, the second integral of (\ref{eq1234}) is bounded by
\begin{equation}\begin{split}\label{abe22}
&\int^{\frac{t}{1-m}}_{0}\int_{\rho(y)<s}\int^{\rho(y)}_{\frac{m\tilde{r}_0}{1-m^2}}\int_{\tilde{L}_{1}}f(\|y-z\|)\left(s^2-\alpha^2\right)^{-1/2}\left(s^2-\rho^2(y)\right)^{-1/2}d\mathcal{A} d\alpha dyds
\end{split}\end{equation}
where $\tilde{L}_{i}=L_i\cap A_{2,2}$ and $\mathcal{A}(L)$ is an arch length of curve $L$. Note that if $z\in \{z|\;\tilde{r}_0<y_1-z_1\}$ then for any fixed $y$ and $\alpha\in[\frac{m\tilde{r}_0}{1-m^2},\rho(y)]$, arch lengths of two curves, $\tilde{L}_1$ and $\tilde{L}_2$ satisfy that for some positive constants C, $\mathcal{A}(\tilde{L}_1)\leq C \mathcal{A}(\tilde{L}_2)$. Therefore, (\ref{abe22}) is bounded by
\begin{equation*}\begin{split}\label{bb120}
&\int^{\frac{t}{1-m}}_{0}\int_{\rho(y)<s}\int^{\rho(y)}_{\frac{m\tilde{r}_0}{1-m^2}}C_1\left(s^2-\rho^2(y)\right)^{-1/2}\left(s^2-\alpha^2\right)^{-1/2}\int_{\tilde{L}_{1}}d\mathcal{A} d\alpha dyds\\
&\leq C_1 \int^{\frac{t}{1-m}}_{0}\int_{\rho(y)<s}\int^{\rho(y)}_{\frac{m\tilde{r}_0}{1-m^2}}\left(s^2-\rho^2(y)\right)^{-1/2}C\int_{\tilde{L}_{2}}\left(s^2-\alpha^2\right)^{-1/2}d\mathcal{A} d\alpha dyds\\
&\leq C\int^{\frac{t}{1-m}}_{0}\int_{\rho(y)<s}\int^{\rho(y)}_{\frac{m\tilde{r}_0}{1-m^2}}\left(s^2-\rho^2(y)\right)^{-1/2}\int_{\tilde{L}_{2}}\left(s^2-\rho^2(z)\right)^{-1/2}d\mathcal{A} d\alpha dyds\\
&\leq \frac{C}{C_2}\int^{\frac{t}{1-m}}_{0}\int_{\rho(y)<s}\int_{A_{2,3}}f(\|y-z\|)\left(s^2-\rho^2(z)\right)^{-1/2}\left(s^2-\rho^2(y)\right)^{-1/2}dzdyds,
\end{split}\end{equation*}
where $A_{2,3}=\{z|\;0<\rho(z)-\frac{mz_1}{1-m^2}, \frac{m\tilde{r}_0}{1-m^2}<\rho(z)<\rho(y)\}\cap\{z|\;\tilde{r}_0<y_1-z_1\}$. Note that $\{z|\; \frac{m\tilde{r}_0}{1-m^2}<\rho(z)<\rho(y)\}=\{z|\; \frac{m\tilde{r}_0-m(y_1-z_1)}{1-m^2}<\rho(z)-\frac{m(y_1-z_1)}{1-m^2}<\rho(y)-\frac{m(y_1-z_1)}{1-m^2}\}\subset\{z|\; \rho(z)-\frac{m(y_1-z_1)}{1-m^2}<\rho(y)\}$, if $y_1-z_1>0$. Therefore, we have $A_{2,3}\subset A_{2,2}$. By Fubini's theorem, the second integral of (\ref{eq1234}) has the same result as the first integral of (\ref{eq1234}).

Consequently, we have an upper bound of (\ref{eq1234}),
\begin{equation}\begin{split}\label{eq1234567}
\int_{D_1(y,z)}f(\|y-z\|)\int^{t+\frac{m y_1}{1-m^2}}_{\rho(y)}\left(s^2-\rho^2(z)\right)^{-1/2}\left(s^2-\rho^2(y)\right)^{-1/2}dsdzdy.
\end{split}\end{equation}
By setting $\bar{s}=s^2$ and (\ref{ttt}), we have an upper bound of (\ref{eq1234}),
\begin{equation*}\begin{split}
&\int_{D_1(y,z)}f(\|y-z\|)\int^{(t+\frac{m y_1}{1-m^2})^2}_{\rho^2(y)}\frac{1}{2\bar{s}}\left(\bar{s}^2-\left(\rho^2(y)+\rho^2(z)\right)\bar{s}+\rho^2(y)\rho^2(z)\right)^{-1/2}d\bar{s}dzdy\\
&\leq\int_{D_1(y,z)}\frac{f(\|y-z\|)}{2\rho(y)}\left(\int^{\left(t+\frac{m y_1}{1-m^2}\right)^2}_{\rho^2(y)}\left(\bar{s}^2-\left(\rho^2(y)+\rho^2(z)\right)\bar{s}+\rho^2(y)\rho^2(z)\right)^{-1/2}d\bar{s}\right)dzdy\\
&\leq\int_{D_1(y,z)}\frac{f(\|y-z\|)}{2\rho(y)}\left(\ln\Big(\sqrt{\left(\frac{t}{1-m}\right)^2-\rho^2(y)}+\sqrt{\left(\frac{t}{1-m}\right)^2-\rho^2(z)}\Big)^2-\ln\Big(\rho^2(y)-\rho^2(z)\Big)\right)dzdy\\
&\leq\int_{D_1(y,z)}\frac{f(\|y-z\|)}{2\rho(y)}\left(\ln \frac{4t^2}{(1-m)^2}-\ln\left(\rho^2(y)-\rho^2(z)\right)\right)dzdy.
\end{split}\end{equation*}
By Lemma \ref{lem1}, we have an upper bound of the integral (\ref{eq1234}),
\begin{equation}\begin{split}
&\int^{2\sqrt{\frac{1+m}{1-m}}t}_{0}rf(r)\int^{2(1+m)t}_{r\sqrt{1-m^2}}\int^{\pi/2}_{\eta^{-1}\left(\frac{w}{(1-m^2)r}\right)}\left(\ln{\frac{4t^2}{(1-m)^2}}-\ln{\Big(\frac{wr\eta(\theta)}{1-m^2}-r^2\eta^2(\theta)\Big)}\right)\left(\ln{(4(1+m)t)}-\ln{w}\right)d\theta dwdr\\
&=\int^{2\sqrt{\frac{1+m}{1-m}}t}_{0}rf(r)\ln{\frac{1}{r}}\int^{2(1+m)t}_{r\sqrt{1-m^2}}\int^{\pi/2}_{\eta^{-1}\left(\frac{w}{(1-m^2)r}\right)}\left(\ln{(4(1+m)t)}-\ln{w}\right)d\theta dwdr\\
&+\int^{2\sqrt{\frac{1+m}{1-m}}t}_{0}rf(r)\int^{2(1+m)t}_{r\sqrt{1-m^2}}\int^{\pi/2}_{\eta^{-1}\left(\frac{w}{(1-m^2)r}\right)}\left(\ln{\frac{4t^2}{(1-m)^2}}-\ln{(\frac{w\eta(\theta)}{1-m^2}-r\eta^2(\theta))}\right)\left(\ln{(4(1+m)t)}-\ln{w}\right)d\theta dwdr.
\end{split}\end{equation}
By the same argument as in \cite{DF1998} (using the fact $\int-\ln w dw <1$, refer to pages 197-199 of \cite{DF1998}), both inner integrals,
\begin{equation*}\begin{split}
&\int^{2(1+m)t}_{r\sqrt{1-m^2}}\int^{\pi/2}_{\eta^{-1}\left(\frac{w}{(1-m^2)r}\right)}\Big(\ln{(4(1+m)t)}-\ln{w}\Big)\;d\theta dw,\\
&\int^{2(1+m)t}_{r\sqrt{1-m^2}}\int^{\pi/2}_{\eta^{-1}\left(\frac{w}{(1-m^2)r}\right)}\Big(\ln{\frac{4t^2}{(1-m)^2}}-\ln{(\frac{w\eta(\theta)}{1-m^2}-r\eta^2(\theta))}\Big)\Big(\ln{(4(1+m)t)}-\ln{w}\Big)\;d\theta dw
\end{split}\end{equation*}
are finite. Therefore, we conclude that
\begin{equation*}\begin{split}
E[|X(t,x)|^2]<C_3 \int^{2\sqrt{\frac{1+m}{1-m}}t}_{0}rf(r)\left(\ln{\frac{1}{r}}+C_4\right)dr,
\end{split}\end{equation*}
where $C_{3}$ and $C_4$ are positive constants depending on $m$.\\

On the other hand, by change of variable $s=\tilde{s}+\frac{my_1}{1-m^2}$, we obtain
\begin{equation}\begin{split}\label{eq12222a}
&\int_{\cup_{i=2,3}D_i(y,z)}f(\|y-z\|)\int^{t}_{\rho(y)-\frac{m y_1}{1-m^2}}G(\tilde{s},y,m)G(\tilde{s},y,m)\;d\tilde{s}dzdy\\
&=\int_{\cup_{i=2,3}D_i(y,z)}f(\|y-z\|)\int^{t+\frac{m y_1}{1-m^2}}_{\rho(y)}\left(s^2-\rho^2(y)\right)^{-1/2}\left(\Big(s+\frac{m (z_1-y_1)}{1-m^2}\Big)^2-\rho^2(z)\right)^{-1/2}dsdzdy.
\end{split}\end{equation}
Case 2 ($z_1>y_1$ and $\rho(y)>\rho(z)$) : we easily obtain an upper bound of (\ref{eq12222a}),
\begin{equation*}\begin{split}
\int_{D_2(y,z)}f(\|y-z\|)\int^{t+\frac{m y_1}{1-m^2}}_{\rho(y)}\left(s^2-\rho^2(y)\right)^{-1/2}\left(s^2-\rho^2(z)\right)^{-1/2}dsdzdy.
\end{split}\end{equation*}
Case 3 ($z_1>y_1$ and $\rho(y)<\rho(z)$) : let $\tilde{s}=s+\epsilon$, $\epsilon=\rho(z)-\rho(y)$. By $0<\frac{m (z_1-y_1)}{1-m^2}-\epsilon$, we have an upper bound of (\ref{eq12222a}),
\begin{equation}\begin{split}\label{abe331}
&\int_{D_3(y,z)}f(\|y-z\|)\int^{t+\frac{m y_1}{1-m^2}+\epsilon}_{\rho(z)}\left((s-\epsilon)^2-\rho^2(y)\right)^{-1/2}\left(s^2-\Big(\rho(z)-\Big(\frac{m (z_1-y_1)}{1-m^2}-\epsilon\Big)\Big)^2\right)^{-1/2}dsdzdy.
\end{split}\end{equation}
By defining $\beta(y,z):=\frac{m (z_1-y_1)}{1-m^2}-2\epsilon$, we obtain $\rho(z)-\Big(\frac{m (z_1-y_1)}{1-m^2}-\epsilon\Big)=\rho(y)-\beta(y,z)$. Note that $\beta(y,z)$ is not always positive. Therefore, we have to check the cases, $\beta(y,z)< 0$ and $\beta(y,z)>0$ separately. If $\beta(y,z)> 0$, then $\rho(z)-(\frac{m (z_1-y_1)}{1-m^2}-\epsilon)=\rho(y)-\beta(y,z)< \rho(y)$. If $\beta(y,z)< 0$, we choose a constant $\tilde{r}_0>0$ such that $\frac{s^2-\rho^2(y)}{s^2-(\rho(y)-\beta(y,z))^2}<4$ for all $y\in \{y|\; -\tilde{r}_0<\beta(y,z)\}$ and use a argument from (\ref{eq1234}) to (\ref{eq1234567}) in the case 1. Then we conclude that (\ref{abe331}) is bounded by
\begin{equation}\begin{split}\label{upp12}
&\int_{D_3(y,z)}f(\|y-z\|)\int^{t+\frac{m y_1}{1-m^2}+\epsilon}_{\rho(z)}\left((s-\epsilon)^2-\rho^2(y)\right)^{-1/2}\left(s^2-\rho^2(y)\right)^{-1/2}dsdzdy.
\end{split}\end{equation}
On the other hand,
\begin{equation*}\begin{split}
(s-\epsilon)^2 - \rho^2(y)=&\left(1-\frac{2\epsilon}{s+\rho(y)+\epsilon}\right)\left(s^2-\left(\epsilon+\rho(y)\right)^2\right)\\
\geq&\frac{\rho(y)}{\rho(z)}\left(s^2-\left(\epsilon+\rho(y)\right)^2\right).
\end{split}\end{equation*}
For $(y,z)\in{D_3(y,z)}$, we obtain $(1-m)\rho(z)\leq(1+m)\rho(y)$. This leads to $((s-\epsilon)^2-\rho^2(y))^{-1/2}\leq \sqrt{\frac{1+m}{1-m}}(s^2-\rho^2(z))^{-1/2}$. Hence, we have an upper bound of (\ref{upp12})
\begin{equation*}\begin{split}
\int_{D_3(y,z)}f(\|y-z\|)\int^{t+\frac{m z_1}{1-m^2}}_{\rho(z)}\left(s^2-\rho^2(z)\right)^{-1/2}\left(s^2-\rho^2(y)\right)^{-1/2}dsdzdy.
\end{split}\end{equation*}
By Lemma \ref{lem1} and arguments (page 8) in the case 1, we conclude that
\begin{equation}\begin{split}\label{l2l2}
E[|X(t,x)|^2]<C_3 \int^{2\sqrt{\frac{1+m}{1-m}}t}_{0}rf(r)\left(\ln{\frac{1}{r}}+C_4\right)dr.
\end{split}\end{equation}
\end{proof}

The result (\ref{l2l2}) is not sufficient to prove joint measurability of $X(t,x)$ as in \cite{DF1998}. To guarantee joint measurability, we will provide the continuity in $L^2$ in Section 3. Moreover, we will establish H\"older continuity of $X(t,x)$ by Theorem \ref{th2} in Section 3.

\section{H\"older Continuity}
In this section, we study H\"older continuity of $X(t,x)$ by providing the modulus of continuity for $X(t,x)$ which implies that $X(t,x)$ is $L^2$-continuous.
\begin{lem}\label{lem4}
Suppose $0<c<b$ and $a<b<t^2$. Then
\begin{equation}\begin{split}\label{lemin5}
\int^{t}_{\sqrt{b}}&\Big((s^2-b)^{-1/2}-(s^2-a)^{-1/2}\Big)(s^2-c)^{-1/2}ds\leq \frac{1}{2\sqrt{b}}\ln{\Big(1+\frac{b-a}{b-c}+2\sqrt{\frac{b-a}{b-c}}\Big)}.
\end{split}\end{equation}
\end{lem}
\begin{proof} See the proof in appendix.
\end{proof}

\begin{lem}\label{lem5}
Suppose $E[|X(t,x)|^2]<\infty$ for $t\leq t_0$. Then there exist positive constants $C_1$ and $C_2$ depending on the Mach number $m$ such that for small h, $0<h<t_0$, $0<t<t_0$ and $x^{(1)},x^{(2)}\in\mathds{R}^2$ with $\|x^{(1)}-x^{(2)}\|=h$, two mean square norms $E[|X(t,x)-X(t+h,x)|^2]$ and $E[|X(t,x^{(1)})-X(t,x^{(2)})|^2]$ are bounded by
\begin{equation}\label{boundeq1}
C_1\int^{2\sqrt{\frac{1+m}{1-m}}t_0}_{0}rf(r)\int^{2(1+m)t_0}_{r\sqrt{1-m^2}}\int^{\pi/2}_{\eta^{-1}\left(\frac{w}{(1-m^2)r}\right)}\ln{\Big(1+\frac{C_2 t_0 h^{1/2}}{\frac{wr\eta(\theta)}{1-m^2}-r^2\eta^2(\theta)}\Big)}\Big(\ln{(4(1+m)t_0)}-\ln{w}\Big)\;d\theta dwdr.
\end{equation}
\end{lem}
\begin{proof}
First, we consider the time increment case, $E[|X(t+h,x)-X(t,x)|^2]$. By change of variable $y=x-\tilde{y}$,
\begin{equation}\begin{split}\label{y2ky2k}
X&(t,x)-X(t+h,x)\\
=&\int^{t}_{0}\int_{\mathds{R}^2}G(t-\tilde{s},x-\tilde{y},m)F(d\tilde{s},d\tilde{y})-\int^{t+h}_{0}\int_{\mathds{R}^2}G(t-\tilde{s}+h,x-\tilde{y},m)F(d\tilde{s},d\tilde{y})\\
=&\int^{t}_{0}\int_{0<\rho(y)-\frac{my_1}{1-m^2}<t-\tilde{s}}G(t-\tilde{s},y,m)-G(t-\tilde{s}+h,y,m)F(d\tilde{s},dy)\\
&-\int^{t}_{0}\int_{0<t-\tilde{s}<\rho(y)-\frac{my_1}{1-m^2}<t-\tilde{s}+h}G(t-\tilde{s}+h,y,m)F(d\tilde{s},dy)\\
&-\int^{t+h}_{t}\int_{0<\rho(y)-\frac{my_1}{1-m^2}<t-\tilde{s}+h}G(t-\tilde{s}+h,y,m)F(d\tilde{s},dy)\\
=:&Y_1+Y_2+Y_3
\end{split}\end{equation}
Therefore, it is enough to show that $E[|Y_{i}|^2]$, $i=1,2,3$, are bounded by (\ref{boundeq1}). Now, we consider $E[|Y_{1}|^2]$. By change of variable $s=t-\tilde{s}$,
\begin{equation}\begin{split}\label{y1}
E[|Y_1|^2]\leq2\int_{\bigcup^{3}_{i=1}D_i}f(\|y-z\|)\int^{t_0}_{\rho(y)-\frac{my_1}{1-m^2}}G^{(h)}(s,y,m)G^{(h)}(s,z,m)\;dsdydz,
\end{split}\end{equation}
where $G^{(h)}(s,y,m)=G(s,y,m)-G(s+h,y,m)$.

As in the proof of Theorem \ref{th1}, we will use Lemma \ref{lem0} and changes of variable for each cases.\\
Case 1 ($z_1<y_1$) : let $\tau=s+\frac{mz_1}{1-m^2}$. By Lemma \ref{lem0} and eliminating the term $G(s+h,z,m)$ in $(\ref{y1})$, we have an upper bound of $E[|Y_{1}|^2]$,
\begin{equation*}\begin{split}
\int_{D_1(y,z)}f(\|y-z\|)\int^{t_0+\frac{my_1}{1-m^2}}_{\rho(y)}\frac{H(\tau,y)}{\sqrt{\tau^2-(\rho(z)+\frac{m(y_1-z_1)}{1-m^2})^2}}\;d\tau dydz,
\end{split}\end{equation*}
where $H(\tau,y)=\frac{1}{\sqrt{\tau^2-\rho^2(y)}}-\frac{1}{\sqrt{(\tau+h)^2-\rho^2(y)}}>0$. Therefore, $E[|Y_{1}|^2]$ is bounded by
\begin{equation}\begin{split}\label{y111}
&\int_{D_1(y,z)}f(\|y-z\|)\int^{t_0+\frac{my_1}{1-m^2}}_{\rho(y)}\frac{H(\tau,y)}{\sqrt{\tau^2-\rho^2(z)}}\;d\tau dydz.
\end{split}\end{equation}

Suppose $z_1>y_1$. For the case 2 and 3, we define $\tau=s+\frac{my_1}{1-m^2}$. By eliminating $((\tau+h+\frac{m(z_1-y_1)}{1-m^2})^2-\rho^2(z))^{-1/2}$ in $G^{(h)}(\tau-\frac{my_1}{1-m^2},z,m)$ of $(\ref{y1})$, the right-hand side integral of (\ref{y1}) is bounded by
\begin{equation}\begin{split}\label{y11}
&\int_{\cup_{i=2,3}D_i(y,z)}f(\|y-z\|)\int^{t_0+\frac{my_1}{1-m^2}}_{\rho(y)}\frac{H(\tau,y)}{\sqrt{(\tau+\frac{m(z_1-y_1)}{1-m^2})^2-\rho^2(z)}}\;d\tau dydz.
\end{split}\end{equation}

Case 2 ($z_1>y_1$ and $\rho(z)<\rho(y)$) : we can easily check that (\ref{y11}) is bounded by
\begin{equation}\begin{split}\label{y1234}
&\int_{D_2(y,z)}f(\|y-z\|)\int^{t_0+\frac{my_1}{1-m^2}}_{\rho(y)}\frac{H(\tau,y)}{\sqrt{\tau^2-\rho^2(z)}}\;d\tau dydz.
\end{split}\end{equation}

Case 3 ($z_1>y_1$ and $\rho(z)>\rho(y)$) : set $\tilde{\tau}=\tau+\epsilon$, where $\epsilon=\rho(z)-\rho(y)$ and $a=\rho(y)$. the integral (\ref{y11}) is bounded by
\begin{equation}\begin{split}\label{y0112}
\int_{D_3(y,z)}f(\|y-z\|)\int^{t_0+\frac{my_1}{1-m^2}+\epsilon}_{\rho(z)}\frac{H(\tilde{\tau}-\epsilon,y)}{\sqrt{(\tilde{\tau}+\frac{m(z_1-y_1)}{1-m^2}-\epsilon)^2-\rho^2(z)}}\;d\tilde{\tau} dydz.
\end{split}\end{equation}

By elementry calculation, we can check that
\begin{equation*}\begin{split}
&\frac{\frac{1}{\sqrt{(\tilde{\tau}-\epsilon)^2-a^2}}-\frac{1}{\sqrt{(\tilde{\tau}+h-\epsilon)^2-a^2}}}{\frac{1}{\sqrt{\tilde{\tau}^2-(a+\epsilon)^2}}-\frac{1}{\sqrt{(\tilde{\tau}+h)^2-(a+\epsilon)^2}}}=\Phi_1(\tilde{\tau},a,\epsilon,h) \Phi_2(\tilde{\tau},a,\epsilon,h),\\
&\Phi_1(\tilde{\tau},a,\epsilon,h)=\sqrt{\frac{(\tilde{\tau}+a+\epsilon)(\tilde{\tau}+h+a+\epsilon)}{(\tilde{\tau}+a-\epsilon)(\tilde{\tau}+h+a-\epsilon)}},\\
\text{and}\;\;&\Phi_2(\tilde{\tau},a,\epsilon,h)=\frac{\sqrt{(\tilde{\tau}+h-a-\epsilon)(\tilde{\tau}+h+a-\epsilon)}-\sqrt{(\tilde{\tau}-a-\epsilon)(\tilde{\tau}+a-\epsilon)}}{\sqrt{(\tilde{\tau}+h-a-\epsilon)(\tilde{\tau}+h+a+\epsilon)}-\sqrt{(\tilde{\tau}-a-\epsilon)(\tilde{\tau}+a+\epsilon)}}.
\end{split}\end{equation*}

Note that $\Phi_i$, $i=1,2$ are continuous and the domain of $\Phi_i$ is bounded by $(0,C]^4$ for some positive constant $C$. In $(\ref{y0112})$, $a$ and $\epsilon$ are positive and smaller than $\tilde{\tau}$. Suppose $(\tilde{\tau}+a-\epsilon)\rightarrow 0$. By $2a<(\tilde{\tau}+a-\epsilon)$ and $\rho(z)-\frac{mz_1}{1-m^2}<\rho(y)-\frac{my_1}{1-m^2}$, $(\tilde{\tau}+a-\epsilon)\rightarrow 0$ implies $\tilde{\tau}\rightarrow 0$ and $\mathcal{O}(\frac{a+\epsilon}{\tilde{\tau}})=\mathcal{O}(\frac{a-\epsilon}{\tilde{\tau}})=\mathcal{O}(1)$. In doing so, we obtain $|\Phi_1(\tilde{\tau},a,\epsilon,h)|<C_1$, where $C_1>0$. Similarly,
$\sqrt{(\tilde{\tau}+h-a-\epsilon)(\tilde{\tau}+h+a+\epsilon)}-\sqrt{(\tilde{\tau}-a-\epsilon)(\tilde{\tau}+a+\epsilon)}\rightarrow 0$ implies $h\rightarrow 0$. Therefore $\sqrt{(\tilde{\tau}+h-a-\epsilon)(\tilde{\tau}+h+a-\epsilon)}-\sqrt{(\tilde{\tau}-a-\epsilon)(\tilde{\tau}+a-\epsilon)}\rightarrow 0$. By the same argument as of $\Phi_1$, we have $|\Phi_2(\tilde{\tau},a,\epsilon,h)|<C_1$.

These two results imply that $$|\Phi_1(\tilde{\tau},a,\epsilon,h)\Phi_2(\tilde{\tau},a,\epsilon,h)|\leq C_1\;\;\;\;\;$$ in $(\ref{y0112})$. Therefore, $(\ref{y11})$ is bounded by
\begin{equation*}\begin{split}
&\int_{D_3(y,z)}f(\|y-z\|)\int^{t_0+\frac{mz_1}{1-m^2}}_{\rho(z)}\frac{H(\tau,z)}{\sqrt{\tau^2-(\rho(z)-\frac{m(z_1-y_1)}{1-m^2})^2}}\;d\tau dydz.
\end{split}\end{equation*}
By the same way as in Theorem \ref{th1}, we have an upper bound of $(\ref{y1})$,
\begin{equation}\begin{split}\label{y112}
&\int_{D_3(y,z)}f(\|y-z\|)\int^{t_0+\frac{mz_1}{1-m^2}}_{\rho(z)}\frac{H(\tau,z)}{\sqrt{\tau^2-\rho^2(y)}}\;d\tau dydz.
\end{split}\end{equation}

From (\ref{y111}), (\ref{y1234}), and (\ref{y112}), the right-hand side of (\ref{y1}) is bounded. Since $h$ is small enough, $H(\tau,\cdot)$ in (\ref{y111}), (\ref{y1234}) and (\ref{y112}) can be replaced by $\frac{1}{\sqrt{\tau^2-\rho^2(\cdot)}}-\frac{1}{\sqrt{\tau^2-(\rho^2(\cdot)-\frac{2t_0}{1-m}h-h^2)}}$. By applying Lemma 4 in \cite{DF1998} and Lemma \ref{lem4}, we have the following upper bound for $(\ref{y1})$,
\begin{equation*}\begin{split}
&\int_{D_1(y,z)}\frac{f(\|y-z\|)}{2\rho(y)}\ln{\Big(1+\frac{\frac{2t_0}{1-m}h+h^2}{\rho^2(y)-\rho^2(z)}+2\sqrt{\frac{\frac{2t_0}{1-m}h+h^2}{\rho^2(y)-\rho^2(z)}}\Big)}\;dzdy\\
&+\int_{D_2(y,z)}\frac{f(\|y-z\|)}{2\rho(y)}\ln{\Big(1+\frac{\frac{2t_0}{1-m}h+h^2}{\rho^2(y)-\rho^2(z)}+2\sqrt{\frac{\frac{2t_0}{1-m}h+h^2}{\rho^2(y)-\rho^2(z)}}\Big)}\;dzdy\\
&+\int_{D_3(y,z)}\frac{f(\|y-z\|)}{2\rho(z)}\ln{\Big(1+\frac{\frac{2t_0}{1-m}h+h^2}{\rho^2(z)-\rho^2(y)}+2\sqrt{\frac{\frac{2t_0}{1-m}h+h^2}{\rho^2(z)-\rho^2(y)}}+\Big)}\;dzdy.
\end{split}\end{equation*}
Since $\frac{2\frac{t_0}{1-m}h+h^2}{\rho^2(y)-\rho^2(x)}+2\sqrt{\frac{2\frac{t_0}{1-m}h+h^2}{\rho^2(y)-\rho^2(z)}}\leq\frac{C t_0 h^{1/2}}{\rho^2(y)-\rho^2(z)}$ for small $h$, Lemma \ref{lem1} implies that $E[|Y_1|^2]$ is bounded by (\ref{boundeq1}).\\

The upper bound of  $E[|Y_2|^2]$ and $E[|Y_3|^2]$ can be derived as in the case of $Y_1$. By the definition of $Y_2$ in (\ref{y2ky2k}),
\begin{equation*}\begin{split}\label{y2ky2k}
E[|Y_2|^2]&=\int^{t}_{0}\mathop{\int_{s<\rho(y)-\frac{my_1}{1-m^2}<s+h}}_{s<\rho(z)-\frac{mz_1}{1-m^2}<s+h}f(\|y-z\|)G(s+h,y,m)G(s+h,z,m)\;dydzds\\
&=\int^{t+h}_{h}\mathop{\int_{s-h<\rho(z)-\frac{mz_1}{1-m^2}<s}}_{s-h<\rho(y)-\frac{my_1}{1-m^2}<s}f(\|y-z\|)G(s,y,m)G(s,z,m)\;dydzds\\
&=2\int^{t+h}_{h}\int_{s-h<\rho(z)-\frac{mz_1}{1-m^2}<\rho(y)-\frac{my_1}{1-m^2}<s}f(\|y-z\|)G(s,y,m)G(s,z,m)\;dydzds.
\end{split}\end{equation*}
By Fubini's theorem,
\begin{equation}\begin{split}
E[|Y_2|^2]&=2\int_{\Omega_0}\int^{(t+h) \wedge (\rho(z)-\frac{mz_1}{1-m^2}+h)}_{h \vee (\rho(y)-\frac{my_1}{1-m^2})}f(\|y-z\|)G(s,y,m)G(s,z,m)\;dsdydz
\end{split}\end{equation}
which is bounded by
\begin{equation}\begin{split}\label{z1356}
\int_{\Omega_0}f(\|y-z\|)\int^{\rho(z)-\frac{mz_1}{1-m^2}+h}_{\rho(y)-\frac{my_1}{1-m^2}}G(s,y,m)G(s,z,m)\;dsdydz,
\end{split}\end{equation}
where $\Omega_0=\{(y,z)| 0<\rho(y)-\frac{my_1}{1-m^2}-h<\rho(z)-\frac{mz_1}{1-m^2}<\rho(y)-\frac{my_1}{1-m^2}<t+h\}$.

Suppose $y_1>z_1$. Let $\tau=s+\frac{mz_1}{1-m^2}$, then we have an upper bound of the inner integral of (\ref{z1356}),
\begin{equation*}\begin{split}
&\int^{\rho(z)-\frac{m(y_1-z_1)}{1-m^2}+h}_{\rho(y)}\left(\tau^2-\rho^2(y)\right)^{-1/2}\left(\tau^2-\rho^2(z)\right)^{-1/2}d\tau\\
&\leq\frac{1}{2\rho(y)}\ln\frac{\Big(\sqrt{(\rho(z)+\frac{m(y_1-z_1)}{1-m^2}+h)^2-\rho^2(y)}+\sqrt{(\rho(z)+\frac{m(y_1-z_1)}{1-m^2}+h)^2-\rho^2(y)}\Big)^2}{\rho^2(y)-\rho^2(z)}\\
&\leq\frac{1}{2\rho(y)}\ln\frac{(\sqrt{(\rho(y)+h)^2-\rho^2(y)}+\sqrt{(\rho(y)+h)^2-\rho^2(y)})^2}{\rho^2(y)-\rho^2(z)}\\
&\leq\frac{1}{2\rho(y)}\ln\frac{2(\rho(y)+h)^2-\rho^2(y)-\rho^2(z)+2\sqrt{(\rho(y)+h)^4-r^4_m(y)}}{\rho^2(y)-\rho^2(z)}\\
&\leq\frac{1}{2\rho(y)}\ln\Big(1+\frac{\frac{4t_0}{1-m}h+2h^2+2\sqrt{4\rho(y)h^3+6\rho^2(y)h^2+4r^3_m(y)h}}{\rho^2(y)-\rho^2(z)}\Big)\\
&\leq\frac{1}{2\rho(y)}\ln\Big(1+\frac{Ct_0 h^{\frac{1}{2}}}{\rho^2(y)-\rho^2(z)}\Big).
\end{split}\end{equation*}
The case $y_1<z_1$ is also derived in the following. The inner integral of (\ref{z1356}) is bounded by
\begin{equation*}\begin{split}
\frac{1_{\{\rho(z)<\rho(y)\}}}{2\rho(y)}\ln\Big(1+\frac{Ct_0 h^{\frac{1}{2}}}{\rho^2(y)-\rho^2(z)}\Big)+\frac{1_{\{\rho(z)>\rho(y)\}}}{2\rho(z)}\ln\Big(1+\frac{Ct_0 h^{\frac{1}{2}}}{\rho^2(z)-\rho^2(y)}\Big).
\end{split}\end{equation*}
By $\Omega_0\subset\{(y,z)|\;0<\rho(z)-\frac{mz_1}{1-m^2}<\rho(y)-\frac{my_1}{1-m^2}<t+h=\tilde{t}\;\}$ and Lemma \ref{lem1}, $E[|Y_2|^2]$ is bounded by (\ref{boundeq1}).

Next we consider $E[|Y_3|^2]$. Let $s=t+h-\tilde{s}$. Then $E[|Y_3|^2]$ is
\begin{equation}\begin{split}\label{y3}
C\int_{0<\rho(z)-\frac{mz_1}{1-m^2}<\rho(y)-\frac{my_1}{1-m^2}<h}f(\|y-z\|)\int^{h}_{\rho(y)-\frac{my_1}{1-m^2}}\frac{1}{\sqrt{(s+\frac{my_1}{1-m^2})^2-\rho^2(y)}}\frac{1}{\sqrt{(s+\frac{mz_1}{1-m^2})^2-\rho^2(z)}}\;dsdydz.
\end{split}\end{equation}
If $h$ is replaced by $t$, then (\ref{y3}) is the same as the right-hand side of (\ref{ghighi}) in the proof of Theorem \ref{th1}. Therefore, we have upper bounds of the inner integral of (\ref{y3})
\begin{equation*}\begin{split}
&\int^{h+\frac{my_1}{1-m^2}}_{\rho(y)}\left(s^2-\rho^2(y)\right)^{-1/2}\left(s^2-\rho^2(z)\right)^{-1/2}ds,\;\;\;\;if\;\{(y,z)|\;z_1<y_1\}\;\text{or}\; \{(y,z)|\;z_1>y_1, \rho(y)>\rho(z)\},\\
&\int^{h+\frac{my_1}{1-m^2}}_{\rho(z)}\left(s^2-\rho^2(y)\right)^{-1/2}\left(s^2-\rho^2(z)\right)^{-1/2}ds,\;\;\;\;if\;\{(y,z)|\;z_1>y_1, \rho(y)<\rho(z)\}.
\end{split}\end{equation*}
If $\{(y,z)|\;z_1<y_1\}$ or $\{(y,z)|\;z_1>y_1, \rho(y)>\rho(z)\}$, we have
\begin{equation*}\begin{split}
&\int^{h+\frac{my_1}{1-m^2}}_{\rho(y)}\left(s^2-\rho^2(y)\right)^{-1/2}\left(s^2-\rho^2(z)\right)^{-1/2}ds\\
&\leq\frac{1}{2\rho(y)}\ln\frac{\left(\sqrt{(h+\frac{my_1}{1-m^2})^2-\rho^2(y)}+\sqrt{(h+\frac{my_1}{1-m^2})^2-\rho^2(z)}\right)^2}{\rho^2(y)-\rho^2(z)}\\
&\leq\frac{1}{2\rho(y)}\ln\frac{\left(\sqrt{(\frac{h}{1-m})^2-\rho^2(y)}+\sqrt{(\frac{h}{1-m})^2-\rho^2(z)}\right)^2}{\rho^2(y)-\rho^2(z)}\\
&\leq\frac{1}{2\rho(y)}\ln\frac{4(\frac{h}{1-m})^2-2(\rho^2(y)+\rho^2(z))}{\rho^2(y)-\rho^2(z)}\\
&\leq\frac{1}{2\rho(y)}\ln\frac{4(\frac{h}{1-m})^2+|\rho^2(y)-\rho^2(z)|}{\rho^2(y)-\rho^2(z)}\\
&\leq\frac{1}{2\rho(y)}\ln\Big(1+\frac{4(\frac{h}{1-m})^2}{\rho^2(y)-\rho^2(z)}\Big).
\end{split}\end{equation*}
Similarly, we obtain an upper bound for $\{(y,z)|\;z_1>y_1, \rho(y)<\rho(z)\}$,
\begin{equation*}\begin{split}
\frac{1}{2\rho(z)}\ln\Big(1+\frac{4(\frac{h}{1-m})^2}{\rho^2(z)-\rho^2(z)}\Big).
\end{split}\end{equation*}
Hence, Lemma \ref{lem1} leads to the proof.

Finally, we will show that $E[|X(t,x^{(1)})-X(t,x^{(2)})|^2]$ is bounded by (\ref{boundeq1}). Let $x:=x^{(2)}-x^{(1)}$, $\|x\|=h$. Likewise the time increment case,
\begin{equation*}\begin{split}
X&(t,x^{(1)})-X(t,x^{(2)})\\
=&\int^{t}_{0}\int_{\mathds{R}^2}\left(G(t-s,x^{(1)}-y,m)-G(t-s,x^{(2)}-y,m)\right)F(ds,dy)\\
=&\int^{t}_{0}\mathop{\int_{0<\rho(y)-\frac{my_1}{1-m^2}<s}}_{0<r_m(x+y)-\frac{m(x_1+y_1)}{1-m^2}<s}G(s,y,m)-G(s,x+y,m)F(ds,dy)\\
&+\int^{t}_{0}\mathop{\int_{0<\rho(y)-\frac{my_1}{1-m^2}<s}}_{r_m(x+y)-\frac{m(x_1+y_1)}{1-m^2}>s}G(s,y,m)F(ds,dy)\\
&-\int^{t}_{0}\mathop{\int_{\rho(y)-\frac{my_1}{1-m^2}>s}}_{0<r_m(x+y)-\frac{m(x_1+y_1)}{1-m^2}<s}G(s,x+y,m)F(ds,dy)\\
=:&Z_1+Z_2+Z_3.
\end{split}\end{equation*}
It holds that $E[|Z_{i}|^2]$, $i=2,3$ are bounded by the same upper bound, due to $E[|Z_{2}|^2]=E[|Z_{3}|^2]$. We only derive the upper bound of $E[|Z_2|^2]$. Let $\bar{h}=\frac{h}{1-m}$. Then, by $\bar{h}\geq \rho(x)-\frac{mx_1}{1-m^2}$,
\begin{equation*}\begin{split}
E&[|Z_2|^2]\\
\leq&\int^{t}_{0}\mathop{\int_{0<s-\rho(x)+\frac{mx_1}{1-m^2}<\rho(y)-\frac{my_1}{1-m^2}<s}}_{0<s-\rho(x)+\frac{mx_1}{1-m^2}<\rho(z)-\frac{mz_1}{1-m^2}<s}f(\|y-z\|)G(s,y,m)G(s,z,m)\;dzdyds\\
\leq&\int^{\bar{h}}_{0}\mathop{\int_{0<s-\rho(x)+\frac{mx_1}{1-m^2}<\rho(y)-\frac{my_1}{1-m^2}<s}}_{0<s-\rho(x)+\frac{mx_1}{1-m^2}<\rho(z)-\frac{mz_1}{1-m^2}<s}f(\|y-z\|)G(s,y,m)G(s,z,m)\;dzdyds\\
&+\int^{t}_{\bar{h}}\mathop{\int_{0<s-\rho(x)+\frac{mx_1}{1-m^2}<\rho(y)-\frac{my_1}{1-m^2}<s}}_{0<s-\rho(x)+\frac{mx_1}{1-m^2}<\rho(z)-\frac{mz_1}{1-m^2}<s}f(\|y-z\|)G(s,y,m)G(s,z,m)\;dzdyds\\
\leq&\int^{\bar{h}}_{0}\mathop{\int_{0<\rho(y)-\frac{my_1}{1-m^2}<s}}_{0<\rho(z)-\frac{mz_1}{1-m^2}<s}f(\|y-z\|)G(s,y,m)G(s,z,m)\;dzdyds\\
&+\int^{t+\bar{h}}_{\bar{h}}\mathop{\int_{0<s-\tilde{h}<\rho(y)-\frac{my_1}{1-m^2}<s+\bar{h}}}_{0<s-\tilde{h}<\rho(z)-\frac{mz_1}{1-m^2}<s+\bar{h}}f(\|y-z\|)G(s,y,m)G(s,z,m)\;dzdyds\\
=:&Z_{2,1}+Z_{2,2}
\end{split}\end{equation*}
By replacing $\bar{h}$ with $h$, integrals $Z_{2,1}$ and $Z_{2,2}$ are same as $Y_3$ and $Y_2$ respectively. Similarly, (\ref{boundeq1}) is an upper bound of $E[|Z_{i}|^2]$, $i=2,3$.

By Fubini's theorem, $E[|Z_1|^2]$ is
\begin{equation*}\begin{split}
\int_{\Omega_1}\int_{\Omega_2}\int^{t}_{\tilde{\alpha}}f(\|y-z\|)\left(G(s,y,m)-G(s,x+y,m)\right)\left(G(s,z,m)-G(s,x+z,m)\right)dzdyds.
\end{split}\end{equation*}
where $\tilde{\alpha}=\max\{\rho(z)-\frac{mz_1}{1-m^2}, \rho(y)-\frac{my_1}{1-m^2}, \rho(x+z)-\frac{m(x_1+z_1)}{1-m^2}, \rho(x+y)-\frac{m(x_1+y_1)}{1-m^2}\}$ and
\begin{equation*}\begin{split}
\Omega_1=\{(y,z)|0<\rho(z)-\frac{mz_1}{1-m^2}<t, 0<r_m(x+z)-\frac{m(x_1+z_1)}{1-m^2}<t\},\\
\Omega_2=\{(y,z)|0<\rho(y)-\frac{my_1}{1-m^2}<t, 0<r_m(x+y)-\frac{m(x_1+y_1)}{1-m^2}<t\}.
\end{split}\end{equation*}
The domain $\Omega_1\times \Omega_2$ consists of four disjoint areas:
\begin{equation*}\begin{split}
\bar{D}_1(y,z)=\{(y,z)|\;0<\rho(x+y)-\frac{m(x_1+y_1)}{1-m^2}<\rho(y)-\frac{my_1}{1-m^2}<t, 0<\rho(x+z)-\frac{m(x_1+z_1)}{1-m^2}<\rho(z)-\frac{mz_1}{1-m^2}<t\},\\
\bar{D}_2(y,z)=\{(y,z)|\;0<\rho(y)-\frac{my_1}{1-m^2}<\rho(x+y)-\frac{m(x_1+y_1)}{1-m^2}<t, 0<\rho(z)-\frac{mz_1}{1-m^2}<\rho(x+z)-\frac{m(x_1+z_1)}{1-m^2}<t\},\\
\bar{D}_3(y,z)=\{(y,z)|\;0<\rho(y)-\frac{my_1}{1-m^2}<\rho(x+y)-\frac{m(x_1+y_1)}{1-m^2}<t, 0<\rho(x+z)-\frac{m(x_1+z_1)}{1-m^2}<\rho(z)-\frac{mz_1}{1-m^2}<t\},\\
\bar{D}_4(y,z)=\{(y,z)|\;0<\rho(x+y)-\frac{m(x_1+y_1)}{1-m^2}<\rho(y)-\frac{my_1}{1-m^2}<t, 0<\rho(z)-\frac{mz_1}{1-m^2}<\rho(x+z)-\frac{m(x_1+z_1)}{1-m^2}<t\}.
\end{split}\end{equation*}
By similar arguments in Dalang~\cite{DF1998} (refer to pages 209-210 of \cite{DF1998}), it is sufficient to check the integral over $\bar{D}_1$. By a symmetric property of $\bar{D}_1(y,z)$, we obtain
\begin{equation*}\begin{split}
&\int_{\bar{D}_1(y,z)}f(\|y-z\|)\int^{t}_{\rho(y)-\frac{my_1}{1-m^2}\vee \rho(z)-\frac{mz_1}{1-m^2}}\left(G(s,y,m)-G(s,x+y,m)\right)\left(G(s,z,m)-G(s,x+z,m)\right)\;dzdyds\\
&\leq2\int_{\cup^{3}_{i=1}D_i(y,z)\cap D_4(y)}f(\|y-z\|)\int^{t}_{\rho(y)-\frac{my_1}{1-m^2}}(G(s,y,m)-G(s,x+y,m))(G(s,z,m)-G(s,x+z,m))\;dsdydz,
\end{split}\end{equation*}
where $D_4(y)=\{y|\;0<\rho(x+y)-\frac{m(x_1+y_1)}{1-m^2}<\rho(y)-\frac{my_1}{1-m^2}\}$.

Suppose $x_1\geq0$. Since $x=x^{(2)}-x^{(1)}$, an upper bound of the case $x_1<0$ is the same as of $\tilde{x}:=x^{(1)}-x^{(2)}$ and $\tilde{x}_1>0$. Therefore, the case $x_1\geq0$ is enough to obtain an upper bound of $E[|Z_1|^2]$. Let $\hat{h}=\frac{h}{1-m^2}$ and $\tilde{h}=\frac{mx_1}{1-m^2}$. As in the case of $Y_1$, we have an upper bound of $E[|Z_1|^2]$,
\begin{equation*}\begin{split}
\int_{D_1(y,z)\cap D_4(y)}f(\|y-z\|)\int^{t-\frac{my_1}{1-m^2}}_{\rho(y)} \frac{1}{\sqrt{s^2-\rho^2(z)}}\left(\frac{1}{\sqrt{s^2-\rho^2(y)}}-\frac{1}{\sqrt{(s+\tilde{h})^2-\rho^2(x+y)}}\right)dsdydz\\
+\int_{D_2(y,z)\cap D_4(y)}f(\|y-z\|)\int^{t-\frac{my_1}{1-m^2}}_{\rho(y)} \frac{1}{\sqrt{s^2-\rho^2(z)}}\left(\frac{1}{\sqrt{s^2-\rho^2(y)}}-\frac{1}{\sqrt{(s+\tilde{h})^2-\rho^2(x+y)}}\right)dsdydz\\
+\int_{D_3(y,z)\cap D_4(y)}f(\|y-z\|)\int^{t-\frac{mz_1}{1-m^2}}_{\rho(z)} \frac{1}{\sqrt{s^2-\rho^2(z)}}\left(\frac{1}{\sqrt{s^2-\rho^2(y)}}-\frac{1}{\sqrt{(s+\tilde{h})^2-\rho^2(x+y)}}\right)dsdydz.
\end{split}\end{equation*}
For $x_1\geq0$, it holds that $(s+\tilde{h})^2-\rho^2(x+y)=s^2-(\rho^2(x+y)-\tilde{h}^2-2s\tilde{h})>s^2-\rho^2(y)$ in $D_4(y)$. By Lemma 4 in \cite{DF1998} and Lemma \ref{lem4}, we have an upper bound,
\begin{equation*}\begin{split}
&\int_{D_1(y,z)}\frac{f(\|y-z\|)}{2\rho(y)}\ln{\Big(1+\frac{\frac{4t_0}{1-m}\hat{h}+\hat{h}^2}{\rho^2(y)-\rho^2(z)}+2\sqrt{\frac{\frac{4t_0}{1-m}\hat{h}+\hat{h}^2}{\rho^2(y)-\rho^2(z)}}\Big)}\;dzdy\\
&+\int_{D_2(y,z)}\frac{f(\|y-z\|)}{2\rho(y)}\ln{\Big(1+\frac{\frac{4t_0}{1-m}\hat{h}+\hat{h}^2}{\rho^2(y)-\rho^2(z)}+2\sqrt{\frac{\frac{4t_0}{1-m}\hat{h}+\hat{h}^2}{\rho^2(y)-\rho^2(z)}}\Big)}\;dzdy\\
&+\int_{D_3(y,z)}\frac{f(\|y-z\|)}{2\rho(z)}\ln{\Big(1+\frac{\frac{4t_0}{1-m}\hat{h}+\hat{h}^2}{\rho^2(z)-\rho^2(y)}+2\sqrt{\frac{\frac{4t_0}{1-m}\hat{h}+\hat{h}^2}{\rho^2(z)-\rho^2(y)}}\Big)}\;dzdy.
\end{split}\end{equation*}
By Lemma \ref{lem1}, we finish the proof.
\end{proof}

{\bf Remark 2.} The integrand of the given integration (\ref{boundeq1}) is bounded by the $L^1$-integrable function $$rf(r)\left(\ln{C}-\ln{r}-\ln{(\frac{w\phi(\theta)}{1-m^2}-r\phi^2(\theta))}\right)\left(\ln(4(1+m)t_0)-\ln{w}\right).$$ By dominated convergence theorem, (\ref{boundeq1}) goes to $0$ as $h\rightarrow 0$. \\

\begin{thm}\label{th2}
Assume that there is $0<\alpha< 1$ such that
\begin{equation}
\int_{0^+}f(r)r^{1-\alpha}dr<\infty.
\end{equation}
Then $X$ is H\"older-continuous with $b$, $b\in(0,\frac{\alpha}{4})$.
\end{thm}
\begin{proof}
We apply the same way as in Theorem 3 from \cite{DF1998} as follows. For each $b\in(0,1]$, there is a constant $C$ such that for all $x>0$, $\ln(1+x)\leq C x^{b}$. Therefore (\ref{boundeq1}) is less than or equal to
\begin{equation}\label{boundeq112}
C h^{b/2}\int^{2\sqrt{\frac{1+m}{1-m}}t_0}_{0}rf(r)\int^{2(1+m)t_0}_{r\sqrt{1-m^2}}\int^{\pi/2}_{\eta^{-1}\left(\frac{w}{(1-m^2)r}\right)}(\frac{wr\eta(\theta)}{1-m^2}-r^2\eta^2(\theta))^{-b}\Big(\ln{(4(1+m)t_0)}-\ln{w}\Big)\;d\theta dwdr.
\end{equation}
By replacing $\ln w$ with $\ln (r\sqrt{1-m^2})$ and using the boundness of $\eta$ in the proof of Lemma \ref{lem1}, the inner integral ($\int \cdot \;d\theta dw$) is less than or equal to
\begin{equation}\label{aaa}
C_1 r^{-b}(\ln{(4(1+m)t_0)}-\ln{(r\sqrt{1-m^2})})\frac{(\frac{2t_0}{1-m}-C_2 r)^{1-b}}{1-b}\leq C_1 r^{-b}(\ln{(4(1+m)t_0)}-\ln{(r\sqrt{1-m^2})})\frac{(\frac{2t_0}{1-m})^{1-b}}{1-b}
\end{equation}
for some $C_1,C_2>0$. From (\ref{aaa}), the integral (\ref{boundeq112}) is bounded by
\begin{equation}
C h^{b/2}\int^{2\sqrt{\frac{1+m}{1-m}}t_0}_{0}f(r)r^{1-b}(\ln{(4(1+m)t_0)}-\ln{(r\sqrt{1-m^2})})dr.
\end{equation}
For $b\in(0,\alpha)$, we get $\lim_{r\rightarrow 0}\frac{r^{1-b}(\ln{(4(1+m)t_0)}-\ln{(r\sqrt{1-m^2})})}{r^{1-\alpha}}=0$. This implies that (\ref{boundeq112}) is bounded by
\begin{equation}
C h^{\alpha/2}\int^{2\sqrt{\frac{1+m}{1-m}}t_0}_{0}f(r)r^{1-\alpha}dr\leq C h^{\alpha/2}.
\end{equation}
Lemma \ref{lem5} leads to
\begin{equation*}
E[|X(t,x)-X(t+h,x)|^2+|X(t,x^{(1)})-X(t,x^{(2)})|^2]\leq C h^{\alpha/2}.
\end{equation*}
Since the solution $X(t,x)$ is Gaussian process, we can easily obtain that $p$-moments, $E[|X(t,x)-X(t+h,x)|^p]$ and $E[|X(t,x^{(1)})-X(t,x^{(2)})|^p]$ are bounded by $C h^{p\alpha/4}$ (refer to \cite{DF1998}). Therefore, Komogorov continuity theorem (refer to \cite{W1986}) implies that $X(t,x)$ is H\"older continuous.
\end{proof}

\section{Appendix}
In this section, we prove Lemma $\ref{lem0}$ and Lemma $\ref{lem4}$.\\
{\bf Proof of Lemma $\ref{lem0}$} :
Let $t=s+\epsilon$ . Then we have
\begin{equation*}
\int^{b}_{a-\epsilon}(s^2-c^2)^{-1/2}((s+\epsilon)^2-a^2)^{-1/2}ds=\int^{b+\epsilon}_{a}((t-\epsilon)^2-c^2)^{-1/2}(t^2-a^2)^{-1/2}dt.
\end{equation*}
For $b+\epsilon>t>a$,
\begin{equation*}\begin{split}
(t-\epsilon)^2-c^2=&(t+c+\epsilon -2\epsilon)(t-c-\epsilon)\\
&=\left(1-2\frac{\epsilon}{t+c+\epsilon}\right)(t^2-(c+\epsilon)^2).
\end{split}\end{equation*}
On the other hand, $t+c+\epsilon\geq 4\epsilon$ if $c\geq\epsilon$ and $t+c+\epsilon> 2\epsilon+2c$ if $c<\epsilon$. This implies
\begin{equation*}\begin{split}
(t-\epsilon)^2-c^2 &\geq \min\{\frac{1}{2}, \frac{c}{c+\epsilon}\}(t^2-(c+\epsilon)^2)\\
&\geq \min\{\frac{1}{2}, \frac{c}{a}\}(t^2-(c+\epsilon)^2).
\end{split}\end{equation*}
Therefore, we conclude that
\begin{equation*}\begin{split}
\left((t-\epsilon)^2-c^2\right)^{-1/2}\leq C\left(t^2-(c+\epsilon)^2\right)^{-1/2}= C\left(t^2-\tilde{c}^2\right)^{-1/2}.
\end{split}\end{equation*}
{\bf Proof of Lemma $\ref{lem4}$} :
Let $\tilde{s}=s^2$. By $\int(s^2+a_1s+a_2)^{-1/2}ds=\ln(a_1+2s+2\sqrt{s^2+a_1s+a_2})$, the left-hand side integral of (\ref{lemin5}) is bounded by
\begin{equation*}\begin{split}
\frac{1}{2\sqrt{b}}&\int^{t^2}_{b}\left((\tilde{s}-b)^{-1/2}-(\tilde{s}-a)^{-1/2}\right)(\tilde{s}-c)^{-1/2}d\tilde{s}\\
&=\frac{1}{2\sqrt{b}}\left(\ln{\Big(\frac{\sqrt{t^2-b}+\sqrt{t^2-c}}{\sqrt{t^2-a}+\sqrt{t^2-c}}\Big)^2}-\ln{\frac{b-c}{\left(\sqrt{b-a}+\sqrt{b-c}\right)^2}}\right)\\
&\leq-\frac{1}{2\sqrt{b}}\ln{\frac{b-c}{\left(\sqrt{b-a}+\sqrt{b-c}\right)^2}}\\
&=\frac{1}{2\sqrt{b}}\ln{\Big(1+\frac{b-a}{b-c}+2\sqrt{\frac{b-a}{b-c}}\Big)}.
\end{split}\end{equation*}

\end{document}